\documentclass[11pt,a4paper]{article}
\usepackage[T1]{fontenc}
\usepackage[utf8]{inputenc}
\usepackage{color}
\usepackage{graphicx}
\usepackage{amsfonts}
\usepackage{extarrows}
\usepackage{amsmath,amsthm,amssymb,color}
\usepackage{hyperref}
\usepackage{eepic}
\usepackage{lineno}
\usepackage{enumerate}	
\usepackage{paralist}
\usepackage{cite}
\usepackage{algorithm}
\usepackage{algorithmicx}
\usepackage{algpseudocode}
\usepackage{authblk}
\usepackage{mathtools}
\usepackage{pifont}
\usepackage[numbers,sort&compress]{natbib}

\usepackage{tikz}

\newtheorem{theorem}{Theorem}[section]

\newtheorem{lemma}[theorem]{Lemma}

\theoremstyle{definition}
\newtheorem{definition}[theorem]{Definition}
\newtheorem{claim}[]{\indent Claim}[section]

\newtheorem{conjecture}[theorem]{Conjecture}

\newtheorem{case}{\indent Case}[section]
\newtheorem{subcase}{\indent Subcase}

\newcounter{tbox}

\begin{document}
\title{\bf Longest  cycles and  longest chordless cycles  in $2$-connected  graphs}
\author[1]{Yanan Hu\footnote{Email: hg@sit.edu.cn.}}
	\author[2]{Chengli Li\footnote{Email: lichengli0130@126.com.}}
	\author[2]{Feng Liu\footnote{Email: fengl69@stu.ecnu.edu.cn.}}

	\affil[1]{ School of Science, Shanghai Institute of Technology, Shanghai 201418, China}
	\affil[2]{ Department of Mathematics,
		East China Normal University, Shanghai, 200241, China}
\date{}
\maketitle

\begin{abstract}
Thomassen's chord conjecture from 1976 states that every longest cycle in a $3$-connected graph has a chord. The circumference $c(G)$ and induced circumference $c'(G)$ of a graph $G$ are the length of its longest cycles and the length of its longest chordless cycles, respectively. In $2017$, Harvey proposed a stronger conjecture: 
Every $2$-connected graph $G$ with minimum degree at least $3$  has $c(G)\geq c'(G)+2$. This conjecture implies Thomassen's chord conjecture.
We observe that wheels  are the unique hamiltonian graphs for which the circumference and the induced circumference differ by exactly one. Thus we need only consider non-hamiltonian graphs for Harvey's conjecture. In this paper, we propose a conjecture involving wheels that is equivalent to Harvey's conjecture on non-hamiltonian graphs.  A graph is $\ell$-holed if its all holes have length exactly $\ell$. Furthermore, we prove that Harvey's conjecture holds for $\ell$-holed graphs and graphs  with a small induced circumference. Consequently, Thomassen's conjecture also holds for this two classes of graphs.


\end{abstract}

\section{Introduction}
We consider finite simple graphs and follow the books \cite{West,B-M} for terminology and notations. The {\it order} of a graph is its number of vertices.   Given a graph $G$, we denote by $V(G)$ and $E(G)$ the vertex set and edge set of a graph $G,$ respectively. Denote by $N_{G}(v)$ the neighborhood in $G$ of a vertex $v.$ The {\it join} of  graphs $G$ and $H$, denote by $G\vee H$, is the graph obtained from the disjoint union $G+H$ by adding  edges to join every vertex in $G$  and every vertex in $H.$ 	For an edge $e$ of $G,$ the graph $G/e$ is obtained by contracting the edge $e$ in $G.$  A {\it wheel} is a graph formed by a cycle (called the rim) joining a new vertex (called the center). A {\it hole} in a graph is an induced cycle of length at least $4.$ A graph is {\it $\ell$-holed} if all its holes have length exactly $\ell.$

Thomassen's famous chord conjecture from 1976 is as follows \cite{B-R-A1985,T1989}.

\begin{conjecture}[{\hspace{-0.05em}Thomassen, 1976}] \label{Thomassen-conj}

Every longest cycle in a $3$-connected graph has a chord.
\end{conjecture}

Although this beautiful conjecture currently remains unsolved,  there are many partial results for which Conjecture~\ref{Thomassen-conj} has been proved. See \cite{Tclc1997,Tclc2018,EB2008,KJC2007,XC2003,XC2003klein,JHH2014,CQZ1987} for some previous and recent progress on this conjecture.

\begin{definition}
    \textup{The circumference of a graph $G$, denoted by $c(G)$, is the length of longest cycle in $G$. The induced circumference of a graph $G$, denoted by $c'(G)$, is the length of longest induced cycle in $G$. }
 \end{definition}
	 It is obvious that $c'(G)\leq c(G)$. We denote by $\delta(G)$ the minimum degree of graph $G.$ In \cite{Harvey}, Harvey posed the following conjecture.

\begin{conjecture}[{\hspace{-0.05em}Harvey, 2017}] \label{Conj2}

Let $k\ge 1$ be an integer. If $G$ is a $2$-connected graph with $\delta(G)\ge \lceil k/2+2\rceil$, then $c'(G)\leq c(G)-k$.
\end{conjecture}

 We observe that the wheel is a unique hamiltonian graph for which the circumference and the induced circumference differ by exactly one. Thus we only need to consider whether Harvey's conjecture holds for non-hamiltonian graphs.  Based on this, we make a slight modification to Conjecture~\ref{Conj2} when $k=2$.

\begin{conjecture}\label{Conjnon-H}
    
If $G$ is a non-hamiltonian $2$-connected graph with $\delta(G)\ge 3$, then $c'(G)\leq c(G)-2$.
\end{conjecture}

 We pose the following conjecture which  is equivalent to Conjecture \ref{Conjnon-H}.
\begin{conjecture}\label{Conj-Wheel}
		Let $G$ be a $2$-connected graph with $\delta(G)\geq 3$.  Then $c(G)=c'(G)+1$  if and only if $G$ is a wheel.
\end{conjecture}

In Section $2$, we prove that Conjecture \ref{Conjnon-H} and Conjecture \ref{Conj-Wheel} are equivalent. In Section $3$, we prove that Conjecture \ref{Conjnon-H} holds for $\ell$-holed graphs and graphs  with a small induced circumference.

\section{Equivalence of Conjecture \ref{Conjnon-H} and Conjecture \ref{Conj-Wheel}}\label{Cha}

\begin{proof}[\bf{Proof that Conjecture \ref{Conjnon-H} and Conjecture \ref{Conj-Wheel} are equivalent}] \smallskip
	    Suppose that Conjecture \ref{Conjnon-H} holds. Let $G$ be a $2$-connected graph of order $n$ with $\delta(G)\geq 3$. If $G$ is a wheel, then $c(G)=c'(G)+1$ by the definition of wheel. If $c(G)=c'(G)+1$, by Conjecture \ref{Conjnon-H}, we have that $G$ is hamiltonian. Thus, $c(G)=n$ and  $c'(G)=n-1$. Let $C$ be a longest induced cycle of $G$. Without loss of generality, assume that  $V(G)\setminus V(C)=\{u\}$. Note that $\delta(G)\geq 3$. This implies that $G=C\vee u$. That is, $G$ is a wheel.
		
		Conversely, suppose that Conjecture \ref{Conj-Wheel} holds. To prove Conjecture  \ref{Conjnon-H},  let $G$ be a counterexample to  Conjecture \ref{Conjnon-H} with minimum order. Then $c(G)\leq c'(G)+1$. Note that $c'(G)\leq c(G)$. This implies that 
		\begin{flalign*}\label{fl1}
			c'(G)\leq c(G)\leq c'(G)+1.
		\end{flalign*}
		
		Suppose that $c(G)=c'(G)+1$. By Conjecture \ref{Conj-Wheel}, we obtain that  $G$ is a wheel. That is, $G$ is hamiltonian, a contradiction. Thus, $c(G)=c'(G)$.
		
		Let $C$ be a longest chordless cycle of $G$. For convenience, we may assume that $c(G)=k$ and $C=v_1,v_2,\ldots,v_k,v_1$. 
		\begin{claim}\label{Claimk}
			$k\geq 4$.
		\end{claim}
		Otherwise, suppose $k\leq 3$. This implies that $C$ is a triangle. Let $H$ be a component of $G-C$. Since $G$ is $2$-connected, $|N_C(H)|\geq 2$. This implies that $C\cup H$ contains a cycle with length at least $4$, a contradiction.
		
		By Claim \ref{Claimk} and $C$ is a longest cycle, we have that $N_G(v_i)\cap N_G(v_{i+1})=\emptyset$. Furthermore,  let $G'=G/v_1v_2$ and denote the new vertex by $v^*$. Since $N_G(v_1)\cap N_G(v_2)=\emptyset$ and $\delta(G)\geq 3$. Clearly, we have $\delta(G')\geq 3$.
		
		\begin{claim}
			$G'$ is non-hamiltonian.
		\end{claim}
		Otherwise, suppose $G'$ is hamiltonian. Let $C'$ be a Hamilton cycle in $G'$. Clearly $C'$ can be transformed to a cycle in $G$. We denote this cycle in $G$ by $C_1$. It is obvious that the $C_1$ in $G$ is at least as long as the cycle $C'$ in $G'$. Since $G$ is non-hamiltonian and $G'$ is hamiltonian. We have $c(G')=n-1=c(G)$. This implies that $c'(G)=n-1$. That is, $k=n-1$. Let $V(G)\setminus V(C)=\{u\}$. Thus, $G=C\vee u$. It follows that $G$ is hamiltonian, a contradiction.
		
		\begin{claim}
			$G'$ is $2$-connected.
		\end{claim}
		If  $G'$ has a cut vertex $u$, then we assert that $u=v^*.$  Otherwise, $u$ is a cut vertex of $G,$ a contradiction. It follows that $\{v_1,v_2\}$ is a vertex cut in $G$. Let $H$ be a component in $G-\{v_1,v_2\}$ which contains no vertex of $V(C)\setminus \{v_1,v_2\}$. Since $G$ is $2$-connected, $\{v_1,v_2\}\subseteq N_G(H)$. This implies that $C$ can be transformed into a cycle $C^*$ in $G$ by adding a path in $H$, and $C^*$ is strictly longer than $C$, a contradiction.
		
		Note that $G'$ is a non-hamiltonian $2$-connected graph of order $n-1$ with $\delta(G')\geq 3$. Therefore, $c(G')\geq c'(G')+2$. On the other hand, $C/v_1v_2$ is an induced cycle in $G'$. Thus, $c'(G')\geq k-1$. This implies that $c(G')\geq k+1$.  Let $C{''}$ be a longest cycle in $G'$. Clearly $C^{''}$ can be transformed to a cycle in $G$. We denote this cycle in $G$ by $C_1^{''}$. It is obvious that the $C_1^{''}$ in $G$ is at least as long as the cycle $C^{''}$ in $G'$. This implies that $c(G)\geq k+1$, which contradicts $c(G)=c'(G)=k$. This proves Conjecture~ \ref{Conjnon-H}.  
\end{proof}

\section{Graphs with a small induced circumference  and $\ell$-holed graphs}

Let $G$ be a graph, we denote by $d_G(u, v)$ the distance between two vertices $u$ and $v$ in $G$. Let $S$ be a set of vertices, and let $u$ be a vertex. We define $d(u,S)=\min\{ d(u,v) : v\in S\}$. Let $X,Y\subseteq V(G)$. An $(X,Y)$-path is a path with one endpoint in $X$, the other endpoint in $Y$, and internal vertices not in $X\cup Y$. If $X=\{x\}$, then we use $(x,Y)$-path instead of $(\{x\}, Y)$-path.
	
\begin{theorem}\label{Thm-small-hole}
		Let $G$ be a $2$-connected non-hamiltonian graph with $\delta(G)\geq 3$. If $c'(G)\leq 6$, then $c(G)\geq c'(G)+2$.
	\end{theorem}
	\begin{proof}[\bf{Proof.}]
		It suffices to show that $G$ has a cycle with length at least $c'(G)+2$. Here we only prove that the case of $c'(G)=6$. The other three cases (for $c'(G)=3,4,5$) can be proved by a similar analysis and we omit the proofs.
		
		Let $G$ be a $2$-connected non-hamiltonian graph with $\delta(G)\geq 3$ and  $c'(G)=6$. To the contrary, suppose $G$ is a counterexample of Theorem \ref{Thm-small-hole}, i.e., $c'(G)\leq c(G)\leq c'(G)+1$.
		Now, let  $C=v_1,v_2,\ldots,v_6,v_1$ be a chordless cycle in $G$. Let $L_i=\{u\in V(G):d(u, V(C))=i\}$ for any positive integer $i$.
		\begin{claim}\label{Claim-L2-empty}
			$L_2=\emptyset.$
		\end{claim}
		
		Otherwise, let $u\in L_2$.  Suppose that $u$ has a neighbor $u'$ in $L_2$.  Since $G$ is $2$-connected, there exist two $(\{u,u'\}, V(C))$-paths $P_1,P_2$ such that $V(P_1)\cap V(P_2)=\emptyset$. For convenience, assume that $P_1=u,a_1,\ldots,v_1$ and $P_2=u',b_1,\ldots,v_i$. Note that $P_1,u,u',P_2$ is a path with length at least $5$. However, it follows that either $u,P_1,v_1,v_2,\ldots, v_i,P_2,u',u$ is a cycle with length at least $8$ or $u',P_2,v_i,v_{i+1},\ldots ,v_6,v_1,P_1,u,u'$ is a cycle  with length at least $8$, a contradiction.  Hence, $N(u)\subseteq L_1$.  Since $G$ is $2$-connected, there exist two $(u, V(C))$-paths $P_3,P_4$ such that $V(P_3)\cap V(P_4)=\{u\}$.  Without loss of generality, we may assume that $P_3=u,c_1,\ldots,c_s,v_1$ and $P_4=u,d_1,\ldots ,d_t,v_i$, respectively.  To forbid $u,P_3,v_1,v_2\ldots, v_i,P_4,u$ and $u,P_4,v_i,v_{i+1}\ldots, v_6,v_1,P_3,u$ being a cycle of length greater than $7$, we have that $v_i=v_4$ and $s=t=1$. Since $c'(G)=6$, the cycle $u,c_1,v_1,v_2,v_3,v_4,d_1,u$ has a chord.   Clearly, $N_G(\{c_1,d_1\})\cap \{v_2,v_3\}=\emptyset.$ This implies that $c_1d_1\in E(G)$. Since $\delta(G)\geq 3$, $u$ has a neighbor $u^*$ distinct from $c_1$ and $d_1$. It is obvious that $N_G(u^*)\cap \{c_1,d_1\}=\emptyset$.   Since $u^*\in L_1$, we have that $u^*$ has neighbor in $V(C)$, say $v_j$. If $j\notin \{v_1,v_4\}$, then either $u,c_1,v_1,v_2,\ldots ,v_j,u^*,u$ is a cycle with length at least $8$ or $u,u^*,v_j,v_{j+1},\ldots ,v_6,v_1,c_1,u$  is a cycle with length at least $8$, a contradiction. By symmetry, we may assume that $v_j=v_1$. 
		However, it follows that $u,c_1,d_1,v_4,v_5,v_6,v_1,u^*,u$ is a cycle with length at least $8$, a contradiction.

		By Claim \ref{Claim-L2-empty}, we can deduce that $V(C)$ is a dominating set of $G.$ Let $H_1,\dots, H_t$ be the components of $G-V(C)$. We will discuss two cases.
		
		\begin{case}
			There exists a component $H_i$ such that $N_C(H_i)$ has two consecutive vertices in $C$.
		\end{case}
		Without loss of generality, we assume that $\{v_1,v_2\}\subseteq N_C(H_1)$.
		Suppose that $N_C(H_1)=\{v_1,v_2\}$. As $G$ is $2$-connected, there exist two distinct vertices $v_1'$ and $v_2'$ such that $v_1'\in N_{H_1}(v_1)$ and $v_2'\in N_{H_1}(v_2)$. Let  $P_{v_1'v_2'}$ be a $(v_1',v_2')$-path with interior in $H_1$. Then
		\[
		v_1,v_1',P_{v_1'v_2'},v_2',v_2,v_3,v_4,v_5,v_6,v_1
		\]
		is a cycle with length at least $8$, a contradiction.
		Therefore, we have $\{v_3,v_4\}\cap N_C(H_1)\neq \emptyset$ or $\{v_5,v_6\}\cap N_C(H_1)\neq \emptyset$. By symmetry, we assume that $\{v_3,v_4\}\cap N_C(H_1)\neq \emptyset$.
		
		\begin{subcase}\label{case-4v}
			$\{v_3,v_4\}\subseteq N_C(H_1)$.
		\end{subcase}
		
		Suppose that $\{v_5,v_6\}\subseteq N_C(H_1)$. We claim that $N_{H_1}(C)$ is a single vertex. Otherwise, there exists a $(v_i,v_{i+1})$-path $P_{v_i,v_{i+1}}$ with length at least $3$ and its interior in $V(H_1)$ for some $i\in [6]$,
		and hence $$v_i,P_{v_i,v_{i+1}},v_{i+1},\dots,v_i$$ is a cycle with length at least $8$, a contradiction.
		Hence, $G$ contains a subgraph $W$ isomorphic to $W_6$. Note that $V(G)\setminus V(W)\neq \emptyset$, as $G$ is non-Hamiltonian. Let $v\in V(G)\setminus V(W)$. Then there exist two $(v,V(W))$-paths $Q_1,Q_2$ such that $V(Q_1)\cap V(Q_2)= \emptyset$, as $G$ is $2$-connected.
		It is not hard to find a cycle with length at least $8$ in $V(W)\cup V(Q_1)\cup V(Q_2)$, a contradiction.
		Therefore, we have $v_5\notin N_C(H_1)$ or $v_6\notin N_C(H_1)$. By symmetry, we assume that $v_5\notin N_C(H_1)$. Without loss of generality, we assume that $v_5\in N_C(H_2)$. 
		
		Suppose that $v_4\in N_C(H_2)$. Let $P_{v_4,v_5}$ be a $(v_4,v_5)$-path with interior in $V(H_2)$ and let $P_{v_1,v_2}$ be a $(v_1,v_2)$-path with interior in $V(H_1)$. Then 
		\[
		v_4,P_{v_4,v_5},v_5,v_6,v_1,P_{v_1,v_2},v_2,v_3,v_4
		\]
		is a cycle with length at least $8$, a contradiction.
		Suppose that $v_3\in N_C(H_2)$. Let $P_{v_5,v_3}$ be a $(v_5,v_3)$-path with interior in $V(H_2)$ and let $P_{v_4,v_2}$ be a $(v_4,v_2)$-path with interior in $V(H_1)$. Then 
		\[
		v_5,P_{v_5,v_3},v_3,v_4,P_{v_4,v_2},v_2,v_1,v_6,v_5
		\]
		is a cycle with length at least $8$, a contradiction.
		By a similar analysis, we also have $v_6,v_2\notin N_C(H_2)$.  Since $G$ is $2$-connected, we have $N_C(H_2)=\{v_5,v_1\}$. Furthermore, there exists a $(v_1,v_5)$-path $P_{v_1,v_5}$ with length at least $3$ and its interior in $V(H_2)$. Let $P_{v_2,v_1}$ be a $(v_2,v_1)$-path with interior in $V(H_1)$. Then
		\[
		v_1,P_{v_1,v_5},v_5,v_4,v_3,v_2,P_{v_2,v_1},v_1
		\]
		is a cycle with length at least $8$, a contradiction.
		
		\begin{subcase}
			$v_3\in N_C(H_1)$ and $v_4\notin N_C(H_1)$.
		\end{subcase}
		Without loss of generality, we assume that $v_4\in N_C(H_2)$. 
		Suppose that $v_2\in N_C(H_2)$. Let $P_{v_4,v_2}$ be a $(v_4,v_2)$-path with interior in $V(H_2)$ and let $P_{v_3,v_1}$ be a $(v_3,v_1)$-path with interior in $V(H_1)$. Then
		\[
		v_4,P_{v_4,v_2},v_2,v_3,P_{v_3,v_1},v_1,v_6,v_5,v_4
		\]
		is a cycle with length at least $8$, a contradiction.
		Suppose that $v_3\in N_C(H_2)$. Let $P_{v_4,v_3}$ be a $(v_4,v_3)$-path with interior in $V(H_2)$ and let $P_{v_3,v_2}$ be a $(v_3,v_2)$-path with interior in $V(H_1)$. Then
		\[
		v_4,P_{v_4,v_3},v_3,P_{v_3,v_2},v_2,v_1,v_6,v_5,v_4
		\]
		is a cycle with length at least $8$, a contradiction. By a similar analysis, we have $v_5\notin N_C(H_2)$. Furthermore, we have $\{v_1,v_6\} \not\subset N_C(H_2)$.
		
		Suppose that $v_6\in N_C(H_2)$. Then there exists a $(v_6,v_4)$-path $P_{v_6,v_4}$ with length at least $3$ and its interior in $V(H_2)$. Let $P_{v_2,v_1}$ be a $(v_2,v_1)$-path with interior in $V(H_1)$. Then,
		\[
		v_6,P_{v_6,v_4},v_4,v_3,v_2,P_{v_2,v_1},v_1,v_6
		\]
		is a cycle with length at least $8$, a contradiction.
		Therefore, $N_C(H_2)=\{v_4,v_1\}$.
		
		Suppose that $v_6\in N_C(H_1)$. Let $P_{v_6,v_3}$ be a $(v_6,v_3)$-path with interior in $V(H_1)$ and let $P_{v_1,v_4}$ be a $(v_1,v_4)$-path with interior in $V(H_2)$.
		Then
		\[
		v_6,P_{v_6,v_3},v_3,v_2,v_1,P_{v_1,v_4},v_4,v_5,v_6
		\]
		is a cycle with length at least $8$, a contradiction.
		Without loss of generality, we may assume that $v_6\in N_C(H_3)$.
		
		Suppose that $v_1\in N_C(H_3)$. Let $P_{v_6,v_1}$ be a $(v_6,v_1)$-path with interior in $V(H_3)$ and let $P_{v_1,v_2}$ be a $(v_1,v_2)$-path with interior in $V(H_1)$. Then
		\[
		v_6,P_{v_6,v_1},v_1,P_{v_1,v_2},v_2,v_3,v_4,v_5,v_6
		\]
		is a cycle with length at least $8$, a contradiction. Similarly, we have $v_5\notin N_C(H_3)$.
		Suppose that $v_2\in N_C(H_3)$. Let $P_{v_6,v_2}$ be a $(v_6,v_2)$-path with interior in $V(H_3)$ and let $P_{v_1,v_3}$ be a $(v_1,v_3)$-path with interior in $V(H_1)$. Then
		\[
		v_6,P_{v_6,v_2},v_2,v_1,P_{v_1,v_3},v_3,v_4,v_5,v_6
		\]
		is a cycle with length at least $8$, a contradiction.
		Suppose that $v_3\in N_C(H_3)$. Let $P_{v_6,v_3}$ be a $(v_6,v_3)$-path with interior in $V(H_3)$ and let $P_{v_1,v_4}$ be a $(v_1,v_4)$-path with interior in $V(H_2)$. Then
		\[
		v_6,P_{v_6,v_3},v_3,v_2,v_1,P_{v_1,v_4},v_4,v_5,v_6
		\]
		is a cycle with length at least $8$, a contradiction.
		Therefore, we have $N_C(H_3)=\{v_6,v_4\}$. Then there exists a $(v_6,v_4)$-path $P_{v_6,v_4}$ with length at least $3$ and its interior in $V(H_3)$. Let $P_{v_3,v_2}$ be a $(v_1,v_2)$-path with interior in $V(H_1)$. Then
		\[
  v_6,P_{v_6,v_4},v_4,v_3,P_{v_3,v_2},v_2,v_1,v_6
		\]
		is a cycle with length at least $8$, a contradiction.

		\begin{subcase}
			$v_4\in N_C(H_1)$ and $v_3\notin N_C(H_1)$.
		\end{subcase}
		Without loss of generality, we may assume that $v_3\in N_C(H_2)$.
		Suppose that $v_4\in N_C(H_2)$. Let $P_{v_3,v_4}$ be a $(v_3,v_4)$-path with interior in $V(H_2)$ and let $P_{v_1,v_2}$ be a $(v_1,v_2)$-path with interior in $V(H_1)$. Then
		\[
		v_1,P_{v_1,v_2},v_2,v_3,P_{v_3,v_4},v_4,v_5,v_6,v_1
		\]
		is a cycle with length at least $8$, a contradiction. Similarly, we have $v_2\notin N_C(H_2)$.
		Suppose that $v_1\in N_C(H_2)$. Let $P_{v_3,v_1}$ be a $(v_3,v_1)$-path with interior in $V(H_2)$ and let $P_{v_2,v_4}$ be a $(v_2,v_4)$-path with interior in $V(H_1)$. Then
		\[
		v_3,P_{v_3,v_1},v_1,v_6,v_5,v_4,P_{v_4,v_2},v_2,v_3
		\]
		is a cycle with length at least $8$, a contradiction.
		Suppose that $v_5\in N_C(H_2)$. Let $P_{v_5,v_3}$ be a $(v_5,v_3)$-path with interior in $V(H_2)$ and let $P_{v_2,v_4}$ be a $(v_2,v_4)$-path with interior in $V(H_1)$. Then
		\[
		v_3,P_{v_3,v_5},v_5,v_6,v_1,v_2,P_{v_2,v_4},v_4
		\]
		is a cycle with length at least $8$, a contradiction.
		Therefore, we have $N_C(H_2)=\{v_3,v_6\}$. Furthermore, there exists a $(v_3,v_6)$-path $P_{v_3,v_6}$ with length at least $3$ and its interior in $V(H_2)$. Let $P_{v_2,v_4}$ be a $(v_2,v_4)$-path with interior in $V(H_1)$. Then
		\[
		v_3,P_{v_3,v_6},v_6,v_1,v_2,P_{v_2,v_4},v_4,v_3
		\]
		is a cycle with length at least $8$, a contradiction.
		
		\begin{case}
			For each component $H_i$ of $G-V(C)$, $N_C(H_i)$ is independent in $C$.
		\end{case}
		
		\begin{subcase}
			There exists a component $H_i$ such that $d_C(x,y)=2$ for some $x,y\in N_C(H_i)$.
		\end{subcase}
		Without loss of generality, we may assume that $\{v_1,v_3\}\subseteq N_C(H_1)$ and $v_2\in N_C(H_2)$.
		Suppose that $v_4\in N_C(H_2)$. Let $P_{v_2,v_4}$ be a $(v_2,v_4)$-path with interior in $V(H_2)$ and let $P_{v_1,v_3}$ be a $(v_1,v_3)$-path with interior in $V(H_1)$. Then
		\[
		v_1,P_{v_1,v_3},v_3,v_2,P_{v_2,v_4},v_4,v_5,v_6,v_1
		\]
		is a cycle with length at least $8$, a contradiction. Similarly, we have $v_6\notin N_C(H_2)$.
		Therefore, we have $N_C(H_2)=\{v_2,v_5\}$.
		Furthermore, there exists a $(v_2,v_5)$-path $P_{v_2,v_5}$ with length at least $3$ and its interior in $V(H_2)$. Let $P_{v_3,v_1}$ be a $(v_3,v_1)$-path with interior in $V(H_1)$. Then
		\[
		v_2,P_{v_2,v_5},v_5,v_4,v_3,P_{v_3,v_1},v_1,v_2
		\]
		is a cycle with length at least $8$, a contradiction.
		
		\begin{subcase}
			For any $x,y\in N_C(H_i)$, $d_C(x,y)=3$.
		\end{subcase}
		Without loss of generality, we may assume that $v_1\in N_C(H_1)$ and $v_2\in N_C(H_2)$.
		Then we have $v_4\in N_C(H_1)$ and $v_5\in N_C(H_2)$. Let $P_{v_1,v_4}$ be a $(v_1,v_4)$-path with interior in $V(H_1)$ and let $P_{v_2,v_5}$ be a $(v_2,v_5)$-path with interior in $V(H_2)$.
		Hence,
		\[
		v_1,P_{v_1,v_4},v_4,v_3,v_2,P_{v_2,v_5},v_5,v_6,v_1
		\]
		is a cycle with length at least $8$, a contradiction.
		
		This completes the proof of Theorem \ref{Thm-small-hole}.
	\end{proof}

A chordal graph is a simple graph in which every cycle of length greater than three has a chord.
	If $\ell\ge 4$ is an integer, we say a graph is $\ell$-holed if all its holes have length exactly $\ell$. An $i$-wheel is a wheel such that its hole has $i$ vertices. We say that $X \subseteq V (G)$ is a {\it clique cutset} of $G$ if $G[X]$ is a complete
	graph and $G\setminus X$ is disconnected.  A vertex of $G$ adjacent to all the other vertices of $G$ a {\it universal vertex}.
	
	Before proving our main theorem, we first provide the definitions of two classes of 
	$\ell$-hole graphs: a blow-up of an $\ell$-cycle and a blow-up of an $\ell$-framework. For more information, we recommend that readers refer to \cite{CHPRSSTV2024}.
		
	Let $X, Y$ be disjoint subsets of $V(G),$  $X, Y$ are {\it complete} (to each other) if every vertex of $X$ is adjacent to every vertex in $Y,$ and {\it anticomplete} (to each other) if there are no edges between $X, Y.$ We denote by $G[X,Y]$ the bipartite subgraph of $G$ with vertex set $X\cup Y$ and edge set the set of edges of $G$ between $X, Y.$ A {\it half-graph} is a bipartite graph with no induced two-edge matching. Take orderings $x_1, \ldots, x_m$ and $y_1,\ldots , y_n$ of $X$ and $Y$ respectively. We say $G[X, Y ]$ {\it obeys} these orderings if for all $ i, i',  j, j'$ with $1 \leq  i \leq  i'\leq  m$ and $1 \leq j \leq j'\leq n,$ if $x_{i'}y_{j'}$
	is an edge then $x_iy_j$ is an edge; or, equivalently, each vertex in $Y$ is adjacent to an initial segment of $x_1, \ldots , x_m,$ and each vertex in $X$ is adjacent to an initial segment of $y_1,\ldots, y_n.$ 
	
	If $X, Y, Z$ are disjoint cliques of $G$, we say that $G[X, Y ],$ $ G[X, Z]$ are compatible if $G[X, Y\cup Z]$ is a half-graph.
	If $Y, Z$ are anticomplete, then $G[X, Y ], G[X, Z]$ are compatible if and only if there is no induced four-vertex path in $G[X\cup Y \cup Z]$ with first vertex in $Y ,$ second and third in $X,$ and fourth in $Z.$

	\begin{definition}[\cite{CHPRSSTV2024}]\label{1}
		Let $G$ be a graph with vertex set partitioned into sets $W_1,\dots, W_{\ell},$ with the following properties:
		
		(1) $W_1, \dots, W_{\ell}$ are non-null cliques;
		
		(2) for $1\leq i \leq l,$  $G[W_{i-1}, W_i]$ is a half-graph (reading subscripts modulo $\ell$);
		
		(3) for all distinct $i, j \in \{1,\dots , \ell\},$ if there is an edge between $W_i, W_j$ then $j = i\pm1$ (modulo $\ell$);
		
		(4) for  $1\leq i \leq \ell,$  the graphs $G[W_{i}, W_{i+1}],$  $G[W_{i}, W_{i-1}]$ are compatible.
		
		We call such a graph a blow-up of an $\ell$-cycle.
	\end{definition}

	An {\it arborescence} is a tree with its edges directed in such a way that no two edges have a common head; or equivalently, such that for some vertex $r(T)$ (called the apex), every edge is directed away from $r(T).$ A leaf is a vertex different from the apex, with outdegree zero, and $L(T)$ denotes the set of leaves of the arborescence $T.$

	We first introduce the graph $\ell$-framework.  The cases of
	$\ell$ odd and $\ell$ even are different.
	
	\textbf{Case 1. $\ell$ is odd.}
   There are $k+1$ vertices $a_0, \ldots, a_k$ and $k$ vertices $b_1, \ldots, b_k$. For $1 \leq  i \leq k,$ there is a vertical path $P_i$ of length $\frac{\ell-3}{2}$ between $a_i$, $b_i$. The numbers $0,\dots, k$ break into two intervals $\{0,\ldots  , m\}$ and $\{m+1, \ldots, k\}$. 
	
	Each tent is meant to be an arborescence with the given apex and with set of leaves the
	base of the tent.
	Each of the  upper tents contains one vertex in $\{a_0, \ldots, a_m\}$ called its "apex", and obtains a nonempty interval of $\{a_{m+1}, \ldots, a_k\}$ called its "base". Each of $\{a_{m+1}, \ldots, a_k\}$ belongs to the base of an upper tent. The lower tents do the same with left and right switched. There can be any positive number of tents, but there must be a tent with apex $a_0.$  Possibly $m=0,$ and if so there are no lower tents. The way the upper and lower tents interleave is important; for each upper tent (except the innermost when there is an odd number of tents), the leftmost vertex of its base is some $a_i$, and $b_i$ is the apex of one of the lower tents; and for each lower tent (except the innermost when there is an even number of tents), the rightmost vertex of its base corresponds to the apex for one of the upper tents. 
	
	For each $i \in  \{1, \ldots , m\}$, if $a_{i-1}$ is the apex of an upper tent-arborescence $T_{i-1}$ say, there is a directed
	edge from some nonleaf vertex of $T_{i-1}$ (possibly from $a_{i-1}$) to $a_i$; and if $a_{i-1}$ is not the apex of a tent, there is a directed edge from $a_{i-1}$ to $a_i$. So all these upper tent-arborescences and all the vertices $a_0, \ldots , a_m,$ are connected up in a sequence to form one big arborescence $T$ with apex $a_0$ and
	with set of leaves either $\{a_{m+1}, \ldots  , a_k\}$ or $\{a_m, \ldots , a_k\}.$ There is a directed path of $T$ that contains
	$a_0, a_1,\dots, a_m$ in order, possibly containing other vertices of $T$ between them. Similarly for each $i \in \{m + 1, \ldots  , k-1\},$ if $b_{i+1}$ is an apex of a lower tent-arborescence $S_{i+1},$ there is a directed edge from some nonleaf vertex of $S_{i+1}$ to $b_{i}$, and otherwise there is a directed edge $b_{i+1}b_i$. So similarly
	the lower tent-arborescences, and the vertices $b_{m+1}, \ldots, b_k,$ are joined up to make one arborescence $S$ with apex $b_k$ and with set of leaves either $\{b_1, \ldots , b_m\}$ or $\{b_1, \ldots, b_{m+1}\}.$
	
	Thus this describes a graph in which some of the edges are directed: each directed edge belongs to one of two arborescences $T, S$ and each undirected edge belongs to one of the paths $P_i$. We call such a graph an $\ell$-framework.
	
	Let $T$ be an arborescence. For $v\in V (T)$, let $D_v$ be the set of
	all vertices $w\in L(T)$ for which there is a directed path of $T$ from $v$ to $w$. Let $S$ be a
	tree with $V(S) = L(T)$. We say that $T$ lives in $S$ if for each $v\in V(T)$, the set $D_v$ is the
	vertex set of a subtree of $S$. Let $T,T'$ be arborescences with $L(T) = L(T')$. We say they
	are coarboreal if there is a tree $S$ with $V (S) = L(T) = L(T')$ such that $T, T'$ both live
	in $S$.  Finally, let $T, T'$ be arborescences
	and let $\phi$ be a bijection from $L(T)$ onto $L(T')$. We say that $T, T'$ are coarboreal under $\phi$
	if identifying each vertex of $L(T)$ with its image under $\phi$ gives a coarboreal pair.
	
	\textbf{Case 2. $\ell$ is even.} We have vertices $a_0,  \ldots , a_k$ and $k$ vertices $b_1, \ldots , b_k$, but now there is an extra vertex $b_0.$ There are paths $P_i$ between $a_i, b_i$ of
	length $\ell/2-1$ for $1 \leq i \leq m,$ and length $\ell/2-2$ for $m+1 \leq i \leq k.$  There are upper and lower tents as before, but now all the tents have apex on the left. There must
	be an upper tent with apex $a_0,$ and one with apex $a_m,$ although $m = 0$ is permitted. The upper tents are paired with the lower tents; for each upper tent with base $\{a_i, \ldots , a_j\}$ there is also a lower
	tent with base $\{b_i, \ldots, b_j\},$ and vice versa. But the apexes shift by one; if an upper tent has apex $a_i$, the paired lower tent has apex $b_{i+1}$ (or $b_0$ when $i=m$). An important condition is:
	
	$\bullet$ For each upper tent-arborescence $T_i$ say, with apex $a_i$, the paired lower tent-arborescence $S_{i+1}$ with apex $b_{i+1}$ (or $b_0,$ if $i = m$) must be coarboreal with $T_i$ under the bijection that maps $a_j$ to $b_j$ for each leaf $a_j$ of $T_i$.
	
	As before, for each $i \in \{1,\ldots, m\},$ if $a_{i-1}$ is the apex of an upper tent-arborescence $T_{i-1}$ say, there is a directed edge from some nonleaf vertex of $T_{i-1}$ (possibly from $a_{i-1}$) to 
	$a_i$; and if $a_{i-1}$ is not the apex of a tent, there is a directed edge from $a_{i-1}$ to $a_i$. So the upper tent-arborescences
	are connected up to form an arborescence $T$ with apex $a_0,$ and with set of leaves $\{a_{m+1}, \ldots, a_k\}.$
	Also, for each $i \in \{1, \ldots, m -1\},$ if $b_{i+1}$ is the apex of a lower tent-arborescence $S_{i+1}$ say, there is a directed edge from some nonleaf vertex of $S_{i+1}$ (possibly from $b_{i+1}$) to $b_i$; and if $b_{i+1}$ is not the apex of a tent, there is a directed edge from $b_{i+1}$ to $b_i$. Finally, there is a directed edge from some
	nonleaf vertex of the tent-arborescence $S_0$ with apex $b_0$ (possibly from $b_0$ itself) to $b_m.$ So the lower tent-arborescences are connected up to form an arborescence $S$ with apex $b_0,$ and with set of leaves $\{b_{m+1}, \ldots, b_k\}.$ We call this graph an $\ell$-framework.
	
	The transitive closure $\overrightarrow{T}$ of an arborescence $T$ is the undirected graph with vertex set $V (T)$ in which vertices $u, v$ are adjacent if and only if some directed path of $T$ contains both of $u, v.$  Let $F$ be an $\ell$-framework (here, $\ell$ may be odd or even). Let $P_1,\dots, P_k, T, S$ and so on be as in the definition of
	an $\ell$-framework. Let $D =\overrightarrow{T}\cup \overrightarrow{S} \cup P_1 \cup \cdots \cup P_k.$ Thus $V (D) = V (F),$ and distinct $u, v \in  V (D) $ are $D$-adjacent if either they are adjacent in some $P_i,$ or there is a directed path of one of $S, T$ between
	$u, v.$ We say a graph $G$ is a blow-up of $F$ if
	
	(1) $D$ is an induced subgraph of $G,$ and for each $t\in  V (D)$ there is a clique $W_t$ of $G,$ all pairwise disjoint and with union $V (G);$ $W_t \cap V (D) = \{t\}$ for each $t \in V (D),$ and $W_t = \{t\}$ for each $t \in  V (D) \ V (P_1 \cup  \cdots \cup P_k).$
	
	(2) For each $t \in V (D),$ there is a linear ordering of $W_t$ with first term $t,$ say $(x_1,\dots, x_n)$ where
	$x_1 = t.$ It has the property that for all distinct $t, t' \in  V (D),$ if $t, t'$ are not $D$-adjacent then $W_t, W_{t'}$ are anticomplete, and if $t, t'$ are $D$-adjacent then $G[W_t, W_{t'}]$ obeys the orderings of $W_t, W_{t'}$, and every vertex of $G[W_t, W_{t'}]$ has positive degree. (Consequently, if $t, t'$ are $D$-adjacent then
	$t$ is complete to $W_{t'}$ and vice versa.)

	(3) If $t, t' \in  {a_1, \dots , a_k}$ or $t, t' \in {b_1,\dots, b_k},$ and $t, t'$ are $D$-adjacent, then $W_t$ is complete to $ W_{t'}.$
	
	(4) For each $t \in  V (T),$ if $0 \leq  i \leq m$ and $a_i
	, t$ are $D$-adjacent, then $W_t$ is complete to $W_{a_i}$.  For
	each $t \in V (S),$ if either $\ell$ is odd and $i \in {m + 1,\dots, k},$ or $\ell$  is even and $i \in  {0, \dots, m},$ and $b_i, t$ are $D$-adjacent, then $W_t$ is complete to $W_{b_i}$.
	
	(5) For each upper tent-arborescence $T_j$ with apex $a_j$ say, let $t \in L(T_j )$ and let the path $Q$ of $T$ from $a_0$ to $t$ have vertices
	$a_0=y_1 \cdots y_pa_j z_1  \cdots z_q=t$
	in order. Then $W_t$ is complete to $\{y_1, \dots, y_p, a_j\};$ $W_t$
	is anticomplete to $W(T\setminus V(Q))$; and $G[W_t, {z_1,\dots, z_{q-1}}]$ is a half-graph that obeys the given order of $W_t$ and the order $z_1,\dots, z_{q-1}$. The same holds for lower tent-arborescences with $T, a_0$ replaced by $S, b_0.$
	
	Recently, Cook et al. provided some characterizations for 
	$\ell$-hole graphs. We will use their results to complete our proof.
	\begin{lemma}
		(Cook et al. \cite{CHPRSSTV2024})\label{Main-Lem}
		Let $G$ be a graph with no clique cutset and no universal vertex, and let $\ell\geq 7$. Then $G$ is $\ell$-holed if and only if either $G$ is a blow-up of a cycle of length $\ell$, or $G$ is a blow-up of an $\ell$-framewark.
	\end{lemma}
	
	
The following lemma is well-known and easy to verify.
	
	\begin{lemma}\label{Main-Lem-chordal-edge}
		Let $e$ be an edge of a cycle $C$ in a chordal graph. Then $e$ forms a triangle with a third vertex of $C$.
	\end{lemma}
	
	Now we are ready to state and prove the main result.
	\begin{theorem}\label{Main-Thm}
		Let $G$ be a  non-hamiltonian $2$-connected graph with $\delta(G)\geq 3$. If $G$ is $\ell$-holed, then $\ell+2\leq c(G)$.
	\end{theorem}
	\begin{proof}[\bf{Proof.}]
		
		If $\ell\leq 6$, then $c(G)\geq \ell+2$ by Theorem \ref{Thm-small-hole}.
		Let $\ell\geq 7,$ by Lemma \ref{Main-Lem} we may assume that $G$ admits a clique cutset, or a universal vertex, or $G$ is a blow-up of a cycle of length $\ell$ or $G$ is a blow-up of an $\ell$-framewark. By the definitions of the blow-up of a cycle of length $\ell$, we deduced that it is hamitonian.
  Suppose  $G$  is a blow-up of an $\ell$-framewark. Since $G$ is $2$-connected, we have that the number of vertical path $P_i$ is at least two, i.e., $k\ge 2$, otherwise $a_k$ is a cut vertex.  
		
		Suppose that $\ell$ is odd. 
		By the definition of blow-up of $\ell$-framework, $a_0a_kP_kb_kb_{k-1}\allowbreak P_{k-1}a_{k-1}a_0$ is a cycle with length $\ell$. Moreover, as $\ell\ge 7$,  the length of $P_i$ is at least $2$. Let $a_i'$ be the unique neighbor of $a_i$ in $P_i$ for $1\le i\le k$. Then as $\delta(G)\ge 3$, there exists a clique $W_{a_i'}$ with size at least two containing $a_i'$. Let $a_i''$ be a vertex in $W_{a_i'}$. By the definition of blow-up of $\ell$-framework, we have $a_ia_i''\in E(G)$. Therefore,
		\[
		a_0a_ka_k''a_k'P_k[a_k',b_k]b_kb_{k-1}P_{k-1}[b_{k-1},a_{k-1}']a_{k-1}'a_{k-1}''a_{k-1}a_0
		\]
		is a cycle with length $\ell+2$.
		
		Suppose that $\ell$ is even. 
		By the definition of blow-up of $\ell$-framework, $a_0a_kP_kb_kb_0b_{k-1}P_{k-1}a_{k-1}a_0$ is a cycle with length at least $\ell$. Moreover, as $\ell\ge 7$,  the length of $P_i$ is at least $2$. Let $a_i'$ be the unique neighbor of $a_i$ in $P_i$ for $1\le i\le k$. Then as $\delta(G)\ge 3$, there exists a clique $W_{a_i'}$ with size at least two containing $a_i'$. Let $a_i''$ be a vertex in $W_{a_i'}$. By the definition of blow-up of $\ell$-framework, we have $a_ia_i''\in E(G)$. Therefore,
		\[
		a_0a_ka_k''a_k'P_k[a_k',b_k]b_kb_0b_{k-1}P_{k-1}[b_{k-1},a_{k-1}']a_{k-1}'a_{k-1}''a_{k-1}a_0
		\]
		is a cycle with length at least $\ell+2$.
		 Next, We will consider the remaining situations. Let $G$ be a counterexample to Theorem \ref{Main-Thm} with minimum order. 
		
		\begin{claim}\label{Main-Claim-wheel}
			$G$ is $\ell$-wheel-free.
		\end{claim}
		To the contrary, we may assume that $G$ contains an $\ell$-wheel $W$ as a subgraph. 
		Since $G$ is non-hamiltonian, $V(G)\setminus V(W)\neq \emptyset$. Let $v\in V(G)\setminus V(W)$. Then there exist two $(v,V(W))$-paths $Q_1,Q_2$ such that $V(Q_1)\cap V(Q_2)= \emptyset$, as $G$ is $2$-connected.
		It is not hard to find a cycle with length at least $\ell+2$ in $V(W)\cup V(Q_1)\cup V(Q_2)$, a contradiction.
		
		By Claim \ref{Main-Claim-wheel}, we only need to consider the  remaining case that $G$ admits  a clique cutset. Let $S$ be a minimum clique cutset, and let $H_1,\ldots, H_t$  be the components of $G-S$. Furthermore, Let $H_1$ be a component of $G-S$ such that $G[V(H_1)\cup S]$ contains an $\ell$-hole. Furthermore, let $G'=G[V(H_1)\cup S]$.
		
		\begin{claim}\label{Main-Claim-Clique-cutset}
			$|S|=2$.
		\end{claim}
		Since $G$ is $2$-connected, we have $|S|\geq 2$. Suppose that $|S|\geq 3$. Since $S$ is a minimum clique cutset and $|S|\geq 3$, we have $\delta(G')\geq 3$. Note that $G'$ is $2$-connected with fewer vertices than $G$. If $G'$ is non-hamiltonian, then $c(G')\geq \ell+2$. This implies that $c(G)\geq c(G')\geq \ell+2$, a contradiction. Thus, $G'$ is hamiltonian and $c(G')\leq \ell+1$. Since $G'$ contains an $\ell$-hole, $G'$ is an $\ell$-wheel. This implies that $G$ contains an $\ell$-wheel, which contradicts with Claim \ref{Main-Claim-wheel}. This proves Claim \ref{Main-Claim-Clique-cutset}.
		
		For convenience, we may assume that $S=\{u_1,u_2\}$. Let $C=v_1,v_2,\ldots,v_{\ell},v_1 $ be a $\ell$-hole in $G'$.
		\begin{claim}\label{Main-Claim-common}
			$V(C)\cap \{u_1,u_2\}=\emptyset.$ 
		\end{claim}
		Otherwise,we may up to symmetry assume that $u_1\in V(C)$. If $u_2\in V(G)$, then $u_1u_2\in E(C)$. Since $G$ is $2$-connected and $\delta(G)\geq 3$, there exists a $(u_1,u_2)$-path $P$ with length at least $3$ in $G[V(H_2)\cup\{u_1,u_2\}]$. However, it follows that $G[V(C)\cup V(P)]$ is a cycle with length at least $\ell+2$, a contradiction. Hence, $u_2\notin V(C)$. Let $P'$ be a $(u_2,V(C))$-path of $G'$ such that $u_1\notin V(P')$. Since $G$ is $2$-connected, such a path $P'$ exists. For convenience, we may assume that $u_1=v_1$,  $v_j$ is an endpoint of $P$ and $j\leq \frac{\ell+1}{2}$. Choose  $P'$ with $j$  minimum. If $j\leq 3$, then $v_1,P,u_2,P',v_j,v_{j+1},\ldots,v_{\ell},v_{1}$ is a cycle with length at least $\ell+2$, a contradiction. Thus, $j\geq 4$. To forbid $v_1,v_2,\ldots,v_j,P',u_2,P,v_1$ is a cycle with length $\ell +2$, we have that $v_1,v_2,\ldots,v_j,P',u_2,v_1$ has length at most $\ell-3$. This implies that $G[V(P')\cup \{v_1,\ldots, v_j\}]$ is a chordal graph.   By Lemma \ref{Main-Lem-chordal-edge}, the edge $v_1v_2$ forms a triangle with third vertex in $\{v_1,\ldots,v_j\}\cup V(P')$. However, it follows that $N_G(v_2)\cap V(P')\neq \emptyset$, which contradicts with the choice of $j$. This proves that Claim \ref{Main-Claim-common}.
		
		Let $P_1,P_2$ be two disjoint  $(S,V(C))$-paths. Without loss of generality, we may assume that $V(P_1)\cap V(C)=\{v_1\}$ and $V(P_2)\cap V(C)=\{v_j\}$. By symmetry, suppose that $j\leq \frac{\ell+1}{2}$. Furthermore, we may assume that $P_1,P_2$ are two that minimize $|V(P_1)|+|V(P_2)|$ prior to minimizing $j$. By the previous analysis, we have that $j\geq 4$ and $v_1,v_2,\ldots,v_j,P_2,u_2,u_1,P_1,v_1 $ is a cycle with length at most $\ell-1$. This implies that $G[V(P_1)\cup V(P_2)\cup \{v_1,\ldots, v_j\}]$ is a chordal graph.   By Lemma \ref{Main-Lem-chordal-edge}, the edge $v_{j-1}v_j$ forms a triangle with third vertex in $\{v_1,\ldots,v_{j-2}\}\cup V(P_2)\cup V(P_1)$. By the choices of $j$ and $C$ is a chordless cycle, we have $N_G(v_{j-1})\cap V(P_1)\neq \emptyset$, say $x$, and hence
    \[
  v_{j-1},x,P_1[x,u_1],u_1,u_2,P_2,v_j,C[v_j,v_{j-1}],v_{j-1}
     \]
  is a cycle with length at least $\ell+2$, a contradiction.
  
  This completes the proof of Theorem \ref{Main-Thm}.
	\end{proof}

\section*{Acknowledgement}  The authors are grateful to Professor Xingzhi Zhan for his constant support and guidance. This research  was supported by the NSFC grant 12271170 and Science and Technology Commission of Shanghai Municipality (STCSM) grant 22DZ2229014.

	\section*{Declaration}
	
	
	\noindent$\textbf{Conflict~of~interest}$
	The authors declare that they have no known competing financial interests or personal relationships that could have appeared to influence the work reported in this paper.
	
	\noindent$\textbf{Data~availability}$
	Data sharing not applicable to this paper as no datasets were generated or analysed during the current study.

\end{document}